\documentclass{article}
\title{\textbf{The domination number of the $k$-token graph of the complete bipartite graph $K_{m,n}$}}
\author{Nathaniel Chinman, Antonio Cortijo, Beatrice Pedroni, Ian Seong, Lucy Stenger}
\date{}
\usepackage {parskip}
\usepackage [margin=2.5cm] {geometry}
\usepackage{enumitem}
\usepackage[utf8]{inputenc}
\usepackage[T1]{fontenc}
\usepackage{geometry}
\usepackage{amsmath, amssymb, amsthm}
\usepackage{hyperref}
\usepackage{graphicx}
\usepackage{upgreek}
\usepackage{dsfont}
\usepackage{mathtools}
\usepackage{relsize}
\usepackage{bm}
\usepackage{esvect}
\usepackage{circuitikz}
\usepackage{tikz}
\usepackage{tikz-cd}
\usetikzlibrary{positioning}
\usepackage{hhline}
\usepackage{enumitem}
\usepackage{comment}
\usepackage{ytableau}
\usepackage{pdflscape}
\usepackage{mathabx}
\usepackage{multicol}
\usepackage{algorithm2e}
\usepackage{algpseudocode}

\RestyleAlgo{ruled}
\newtheorem{theorem}{Theorem}[section]
\newtheorem{corollary}[theorem]{Corollary}
\newtheorem{lemma}[theorem]{Lemma}

\theoremstyle{definition}
\newtheorem{definition}[theorem]{Definition}

\newtheorem{remark}[theorem]{Remark}

\begin{document}
\maketitle
\begin{abstract}
This paper is about token graphs and domination number. Let $G=(X,E)$ denote an undirected graph that is connected, with no loops or multiple edges, vertex set $X$, and edge set $E$. Pick a positive integer $k\leq \vert X\vert$. The $k$-token graph of $G$, denoted by $F_k(G)$, is defined as follows. The vertex set $\mathcal{X}$ consists of all $k$-subsets of $X$. Two vertices $x,y\in \mathcal{X}$ are adjacent whenever the symmetric difference of $x$ and $y$ consists of two vertices that form an edge in $E$. In this paper, we compute the exact value for the domination number of $F_2(K_{m,n})$ where $K_{m,n}$ is the complete bipartite graph with parameters $m,n$. We present some bounds for the domination number of $F_k(K_{m,n})$ when $k\geq 3$. 
\end{abstract}

\section{Introduction}
\label{intro}
This paper is about token graphs and domination number. To set the stage, we recall the definition of a token graph. Let $G=(X,E)$ denote an undirected graph that is connected, with no loops or multiple edges, vertex set $X$, and edge set $E$. Pick a positive integer $k\leq \vert X\vert$. The $k$-token graph of $G$, denoted by $F_k(G)$, is defined as follows. The vertex set $\mathcal{X}$ consists of all $k$-subsets of $X$. Two vertices $x,y\in \mathcal{X}$ are adjacent whenever the symmetric difference of $x$ and $y$ consists of two vertices that form an edge in $E$. Token graphs were first introduced in \cite{FabFlo} to study multiple objects on the vertices of a graph that move under certain conditions. 

The domination number of $G$, denoted by $\gamma(G)$, is the minimum cardinality of a set $Y\subseteq X$ such that every vertex in $X$ is either in $Y$ or adjacent to some vertex in $Y$. The concept of domination number was first introduced in \cite{Berge}, and the term was first coined in \cite{Ore}. In \cite{FabGom}, the authors found the domination number of the token graph $F_2(G)$ when $G$ is a star graph or a complete graph. For such $G$, the authors also presented bounds for the domination number of the token graph $F_k(G)$ when $k\geq 3$. Token graphs of complete graphs are more commonly known as the Johnson graphs \cite[Section~9.2]{BCN}.

In this paper, we compute the exact value for the domination number of $F_2(K_{m,n})$ where $K_{m,n}$ is the complete bipartite graph with parameters $m,n$. The graph $K_{1,n}$ is a star graph, which is heavily discussed in \cite{FabGom}, so we assume, without loss, that $2\leq m\leq n$. We also present some bounds for the domination number of $F_k(K_{m,n})$ when $k\geq 3$. We now summarize the main results of the paper. For the domination number of $F_2(K_{m,n})$, consider the four cases below:
\begin{enumerate}[label=(C\arabic*)]
    \item $m$ is even;
    \item $m=3$;
    \item $m=n-1\geq 5$ for odd $m$;
    \item $m,n$ satisfy none of (C1)--(C3).
\end{enumerate}

\begin{theorem}
    \label{main1}
    Let $2\leq m\leq n$. Referring to the cases (C1)--(C4) above, the domination number of $F_2(K_{m,n})$ is
    \begin{equation}
        \label{intro1}
        \gamma(F_2(K_{m,n}))=\begin{cases}
            \mathlarger{\frac{m}{2}+n-1}&\text{if $m,n$ satisfy (C1),}\\ \\
            \mathlarger{\frac{m-1}{2}+n-1}&\text{if $m,n$ satisfy one of (C2), (C3)}\\
            \\
            \mathlarger{\frac{m-1}{2}+n}&\text{if $m,n$ satisfy (C4).}
        \end{cases}
    \end{equation}
\end{theorem}

\begin{theorem}
    \label{main2}
    Let $2\leq m\leq n$. Pick an integer $k\geq 3$ such that  $m+n\geq 2k$. The domination number of $F_k(K_{m,n})$ satisfies
    \begin{equation}
    \label{intro2}
        \frac{1}{1+mn} \binom{m+n}{k} \leq \gamma(F_k(K_{m,n}))\leq \frac{n}{m+n-k+1}\binom{m+n}{k}.
    \end{equation}
\end{theorem}
The theorems above are the main results.

This paper is organized as follows. In Section 2, we discuss some preliminary facts on token graphs and domination number. In Section 3, we give an upper bound of $\gamma(F_2(K_{m,n}))$. In Section 4, we give a lower bound of $\gamma(F_2(K_{m,n}))$, and prove Theorem \ref{main1}. In Section 5, we prove  Theorem \ref{main2}. Section 6 is an appendix that illustrates the structure of $F_k(K_{m,n})$.

\section{Prelimininaries}
\label{prelim}
Let $G=(X,E)$ denote an undirected graph that is connected, with no loops or multiple edges, vertex set $X$, and edge set $E$. For a subset $Y\subseteq X$, define
\begin{equation*}
    N[Y]=Y\cup \{x\in X\mid \text{$x$ is adjacent to some vertex in $Y$}\}.
\end{equation*}

\begin{definition}
    We say that a subset $Y\subseteq X$ is a \textit{dominating set of $G$} whenever $N[Y]=X$. The \textit{domination number of $G$}, denoted by $\gamma(G)$, is the minimum cardinality of a dominating set of $G$.
\end{definition}

\begin{lemma}{\rm{\cite{Walikar}}}
\label{DeltaG}
    For a graph $G$, we have 
    \begin{equation}
    \label{deltag}
        \frac{\vert X\vert}{1+\Delta(G)}\leq \gamma(G),
    \end{equation}
    where $\Delta(G)$ is the maximum degree of $G$. 
\end{lemma}

\begin{definition}
    We say that a subset $Y\subseteq X$ is an \textit{independence set of $G$} whenever no pair of vertices in $Y$ are adjacent. The \textit{independence number of $G$}, denoted by $\alpha(G)$, is the greatest cardinality of an independence set of $G$.
\end{definition}

\begin{lemma}{\rm{\cite[p.~29]{Haynes}}}
\label{gammaalpha}
    For a graph $G$, we have $\gamma(G)\leq \alpha(G)$.
\end{lemma}

Next we define token graphs of $G$.

\begin{definition}
    Pick a positive integer $k\leq \vert X\vert$. The $k$-token graph of $G$, denoted by $F_k(G)$, is defined as follows. The vertex set $\mathcal{X}$ of $F_k(G)$ consists of all $k$-subsets of $X$. Two vertices $x,y\in \mathcal{X}$ are adjacent whenever the symmetric difference of $x$ and $y$ consists of two vertices that form an edge in $E$.
\end{definition}
Observe that 
    \begin{equation}
    \label{vertcalX}
        \vert \mathcal{X}\vert=\binom{\vert X\vert}{k}.
    \end{equation}
    Observe that the graphs $F_k(G)$ and $F_{\vert X\vert-k}(G)$ are isomorphic. Observe that for two isomorphic graphs $G_1$ and $G_2$, the token graphs $F_k(G_1),F_k(G_2)$ are isomorphic. 

\begin{lemma}{\rm{\cite[Theorem~2.1]{Abdel}}}
    For a positive integer $k\leq \frac{\vert X\vert}{2}$ we have
    \begin{equation}
    \label{alphaG}
        \alpha (F_k(G))\leq\frac{1}{k}\binom{\vert X\vert}{k-1}\alpha(G).
    \end{equation}
\end{lemma}

\begin{lemma}
\label{importantxy}
    Pick an integer $k\geq 2$. Let $S$ denote a dominating set of $F_k(G)$. Then the set
    \begin{equation}
    \label{nothing}
        \{x\in X\mid \text{no element in $S$ contains $x$}\}
    \end{equation}
    has cardinality at most $1$.
\end{lemma}
\begin{proof}
    Assume that distinct $x,y$ are contained in the set (\ref{nothing}). Then a vertex that contains both $x,y$ is not contained in $N[S]$, which contradicts the assumption that $S$ is a dominating set. The result follows.
\end{proof}

Next we define a complete bipartite graph.
\begin{definition}
\label{bipartitedef}
    For positive integers $m,n$ the complete bipartite graph $K_{m,n}$ is the complement of the disjoint union of complete graphs $K_{m}$ and $K_{n}$. Let $A=\{a_1,a_2,\dots, a_m\}$ (resp. $B=\{b_1,b_2,\dots, b_n\}$) denote the vertex set of $K_{m}$ (resp. $K_{n}$).
\end{definition}

Observe that the graphs $K_{m,n}$ and $K_{n,m}$ are isomorphic. We also observe that the graph $K_{1,n}$ is isomorphic to a star graph, which is heavily discussed in \cite{FabGom}. Throughout the paper, we assume that $G=K_{m,n}$, where $2\leq m\leq n$.

\begin{lemma}{\rm{\cite[Section~3]{West}}}
    Let $2\leq m\leq n$. We have
    \begin{equation}
    \label{alphaK}
        \alpha(K_{m,n})=n.
    \end{equation}
\end{lemma}

Using (\ref{vertcalX}) and the fact that $\vert X\vert=m+n$, we obtain
\begin{equation}
\label{bipartitecalX}
    \vert \mathcal{X}\vert=\binom{m+n}{k}.
\end{equation}

\begin{definition}
    Let $2\leq m\leq n$. Pick a positive integer $k$ such that $m+n\geq 2k$. For $0\leq \ell\leq \min (m,k)$ let the set $\mathcal{C}_{\ell}$ consist of the vertices in $\mathcal{X}$ that contain exactly $\ell$ elements of $A$.

    For convenience, define $\mathcal{C}_{-1}=\emptyset$ and $\mathcal{C}_{\min(m,k)+1}=\emptyset$. Observe that the sets $\{\mathcal{C}_{\ell}\}_{\ell=0}^{\min(m,k)}$ form a partition of $\mathcal{X}$.
\end{definition}

\begin{lemma}
    Let $2\leq m\leq n$. Pick a positive integer $k$ such that $m+n\geq 2k$. For $0\leq \ell\leq \min(m,k)$ the cardinality of the set $\mathcal{C}_{\ell}$ is
    \begin{equation}
    \label{vertC}
        \vert \mathcal{C}_{\ell}\vert=\binom{m}{\ell}\binom{n}{k-\ell}.
    \end{equation}
\end{lemma}

\begin{proof}
    By combinatorial counting.
\end{proof}

\begin{lemma}
\label{threelemmas}
    Let $2\leq m\leq n$. Pick a positive integer $k$ such that $m+n\geq 2k$. Pick an integer $\ell$ where $0\leq \ell\leq \min(m,k)$. For $\chi\in \mathcal{C}_{\ell}$ the following (i)--(iii) hold:
    \begin{enumerate}[label=(\roman*)]
        \item there are $\ell(n-k+\ell)$ vertices in $\mathcal{C}_{\ell-1}$ that are adjacent to $\chi$;

        \item there are $(m-\ell)(k-\ell)$ vertices in $\mathcal{C}_{\ell+1}$ that are adjacent to $\chi$;
        
        \item the degree of $\chi$ is 
    \begin{equation}
    \label{degchi}
        \ell(n-k+\ell)+(m-\ell)(k-\ell).
    \end{equation}
    \end{enumerate}
\end{lemma}

\begin{proof}
    (i) and (ii) follow from combinatorial counting. (iii) follows from (i), (ii), and the fact that any vertex adjacent to $\chi$ is contained in $\mathcal{C}_{\ell-1}$ or $\mathcal{C}_{\ell+1}$.
\end{proof}

\begin{remark}
\label{remark}
    Let $2\leq m\leq n$. Pick a positive integer $k$ such that $m+n\geq 2k$. The partition
    \begin{equation*}
    \mathcal{X}=\bigcup_{\ell=0}^{\min(m,k)}\mathcal{C}_{\ell}
    \end{equation*}
    is equitable in the sense of \cite[p.~159]{Schwenk}. See Figure \ref{class} in the appendix for the illustration of this partition.
\end{remark}

\section{An upper bound of $\gamma(F_2(K_{m,n}))$}
We continue to work with $G=K_{m,n}$, where $2\leq m\leq n$. In this section we present an upper bound on the domination number of $F_2(K_{m,n})$. To do this, we consider each of the cases (C1)--(C4) from Section \ref{intro}.

We first consider the case (C1).
\begin{lemma}
\label{evenupper}
    For even $m$ and $n\geq m$, we have
    \begin{equation*}
        \gamma(F_2(K_{m,n}))\leq \frac{m}{2}+n-1.
    \end{equation*}
\end{lemma}

\begin{proof}
    Assume that $n=m$. The set
    \begin{equation}
    \label{dom4m1}
        \bigcup_{i=1}^{m/2}\{\{a_{2i-1},a_{2i}\}\}\cup \bigcup_{j=1}^{m-1} \{\{a_j,b_j\}\} 
    \end{equation}
    is a dominating set of $F_2(K_{m,m})$ that has size $\dfrac{m}{2}+m-1$. 
    
    Assume that $n=m+1$. The set 
    \begin{equation}
    \label{dom4m2}
        \bigcup_{i=1}^{m/2}\{\{a_{2j-1},a_{2j}\}\}\cup \bigcup_{j=1}^{m} \{\{a_i,b_i\}\}
    \end{equation}
    is a dominating set of $F_2(K_{m,m+1})$ that has size $\dfrac{m}{2}+m$.

    Assume that $n\geq m+2$. The set 
    \begin{equation}
    \label{dom4m3}
        \bigcup_{h=1}^{m/2}\{\{a_{2h-1},a_{2h}\}\}\cup \bigcup_{i=1}^{m} \{\{a_i,b_i\}\}\cup \bigcup_{j=m+1}^{n-1} \{\{a_m,b_j\}\} 
    \end{equation}
    is a dominating set of $F_2(K_{m,n})$ that has size $\dfrac{m}{2}+n-1$. The result follows.
\end{proof}
We illustrate the sets (\ref{dom4m1})--(\ref{dom4m3}) for $m\geq 4$ in Figure \ref{fig:dom4m}.

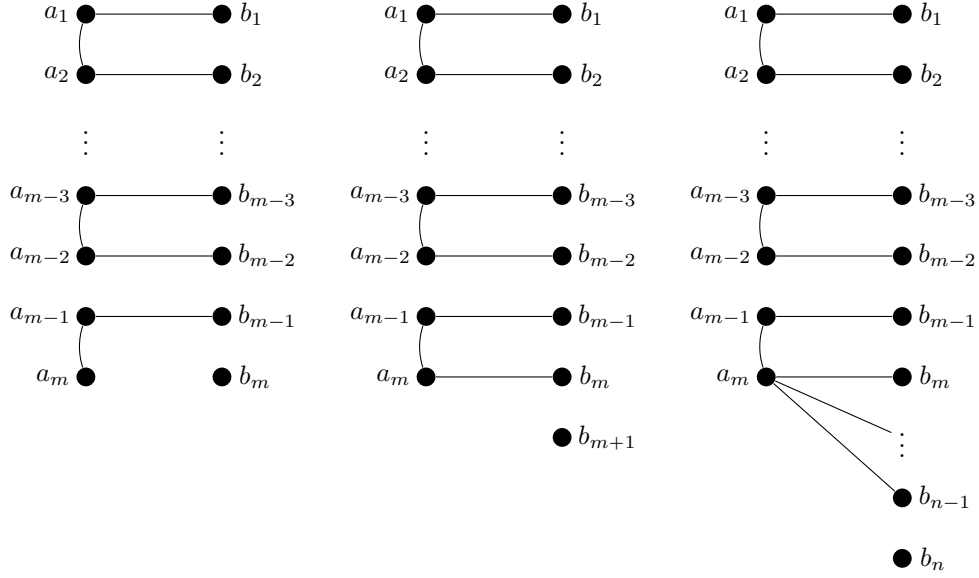
\begin{figure}[!ht]
    \centering
        \begin{tikzpicture}[
    every node/.style={circle, fill=black, inner sep=2.5pt},
    >=stealth
]

\node (a11) at (0,2.8) {};
\node (a12) at (0,2.0) {};
\node[draw=none, fill=none, shape=rectangle] (vdots1) at (0, 1.2) {$\vdots$};
\node (a1m-3) at (0,0.4) {};
\node (a1m-2) at (0,-0.4) {};
\node (a1m-1) at (0,-1.2) {};
\node (a1m) at (0,-2) {};

\node (a21) at (4.5,2.8) {};
\node (a22) at (4.5,2.0) {};
\node[draw=none, fill=none, shape=rectangle] (vdots3) at (4.5, 1.2) {$\vdots$};
\node (a2m-3) at (4.5,0.4) {};
\node (a2m-2) at (4.5,-0.4) {};
\node (a2m-1) at (4.5,-1.2) {};
\node (a2m) at (4.5,-2) {};

\node (a31) at (9,2.8) {};
\node (a32) at (9,2.0) {};
\node[draw=none, fill=none, shape=rectangle] (vdots5) at (9, 1.2) {$\vdots$};
\node (a3m-3) at (9,0.4) {};
\node (a3m-2) at (9,-0.4) {};
\node (a3m-1) at (9,-1.2) {};
\node (a3m) at (9,-2) {};

\node (b11) at (1.8,2.8) {};
\node (b12) at (1.8, 2.0) {};
\node[draw=none, fill=none, shape=rectangle] (vdots2) at (1.8, 1.2) {$\vdots$};
\node (b1m-3) at (1.8,0.4) {};
\node (b1m-2) at (1.8,-0.4) {};
\node (b1m-1) at (1.8,-1.2) {};
\node (b1m) at (1.8,-2) {};

\node (b21) at (6.3,2.8) {};
\node (b22) at (6.3, 2.0) {};
\node[draw=none, fill=none, shape=rectangle] (vdots4) at (6.3, 1.2) {$\vdots$};
\node (b2m-3) at (6.3,0.4) {};
\node (b2m-2) at (6.3,-0.4) {};
\node (b2m-1) at (6.3,-1.2) {};
\node (b2m) at (6.3,-2) {};
\node (b2m+1) at (6.3,-2.8) {};

\node (b31) at (10.8,2.8) {};
\node (b32) at (10.8, 2.0) {};
\node[draw=none, fill=none, shape=rectangle] (vdots6) at (10.8, 1.2) {$\vdots$};
\node (b3m-3) at (10.8,0.4) {};
\node (b3m-2) at (10.8,-0.4) {};
\node (b3m-1) at (10.8,-1.2) {};
\node (b3m) at (10.8,-2) {};
\node[draw=none, fill=none, shape=rectangle] (vdots7) at (10.8, -2.8) {$\vdots$};
\node (b3n-1) at (10.8,-3.6) {};
\node (b3n) at (10.8,-4.4) {};
\node[fill=none, draw=none, inner sep=0pt, left=5pt]  at (a11) {$a_1$};
\node[fill=none, draw=none, inner sep=0pt, left=5pt]  at (a12) {$a_2$};
\node[fill=none, draw=none, inner sep=0pt, left=5pt]  at (a1m-3) {$a_{m-3}$};
\node[fill=none, draw=none, inner sep=0pt, left=5pt]  at (a1m-2) {$a_{m-2}$};
\node[fill=none, draw=none, inner sep=0pt, left=5pt]  at (a1m-1) {$a_{m-1}$};
\node[fill=none, draw=none, inner sep=0pt, left=5pt]  at (a1m) {$a_m$};

\node[fill=none, draw=none, inner sep=0pt, left=5pt]  at (a21) {$a_1$};
\node[fill=none, draw=none, inner sep=0pt, left=5pt]  at (a22) {$a_2$};
\node[fill=none, draw=none, inner sep=0pt, left=5pt]  at (a2m-3) {$a_{m-3}$};
\node[fill=none, draw=none, inner sep=0pt, left=5pt]  at (a2m-2) {$a_{m-2}$};
\node[fill=none, draw=none, inner sep=0pt, left=5pt]  at (a2m-1) {$a_{m-1}$};
\node[fill=none, draw=none, inner sep=0pt, left=5pt]  at (a2m) {$a_m$};

\node[fill=none, draw=none, inner sep=0pt, left=5pt]  at (a31) {$a_1$};
\node[fill=none, draw=none, inner sep=0pt, left=5pt]  at (a32) {$a_2$};
\node[fill=none, draw=none, inner sep=0pt, left=5pt]  at (a3m-3) {$a_{m-3}$};
\node[fill=none, draw=none, inner sep=0pt, left=5pt]  at (a3m-2) {$a_{m-2}$};
\node[fill=none, draw=none, inner sep=0pt, left=5pt]  at (a3m-1) {$a_{m-1}$};
\node[fill=none, draw=none, inner sep=0pt, left=5pt]  at (a3m) {$a_m$};

\node[fill=none, draw=none, inner sep=0pt, right=5pt] at (b11) {$b_1$};
\node[fill=none, draw=none, inner sep=0pt, right=5pt] at (b12) {$b_2$};
\node[fill=none, draw=none, inner sep=0pt, right=5pt] at (b1m-3) {$b_{m-3}$};
\node[fill=none, draw=none, inner sep=0pt, right=5pt] at (b1m-2) {$b_{m-2}$};
\node[fill=none, draw=none, inner sep=0pt, right=5pt] at (b1m-1) {$b_{m-1}$};
\node[fill=none, draw=none, inner sep=0pt, right=5pt] at (b1m) {$b_m$};

\node[fill=none, draw=none, inner sep=0pt, right=5pt] at (b21) {$b_1$};
\node[fill=none, draw=none, inner sep=0pt, right=5pt] at (b22) {$b_2$};
\node[fill=none, draw=none, inner sep=0pt, right=5pt] at (b2m-3) {$b_{m-3}$};
\node[fill=none, draw=none, inner sep=0pt, right=5pt] at (b2m-2) {$b_{m-2}$};
\node[fill=none, draw=none, inner sep=0pt, right=5pt] at (b2m-1) {$b_{m-1}$};
\node[fill=none, draw=none, inner sep=0pt, right=5pt] at (b2m) {$b_m$};
\node[fill=none, draw=none, inner sep=0pt, right=5pt] at (b2m+1) {$b_{m+1}$};

\node[fill=none, draw=none, inner sep=0pt, right=5pt] at (b31) {$b_1$};
\node[fill=none, draw=none, inner sep=0pt, right=5pt] at (b32) {$b_2$};
\node[fill=none, draw=none, inner sep=0pt, right=5pt] at (b3m-3) {$b_{m-3}$};
\node[fill=none, draw=none, inner sep=0pt, right=5pt] at (b3m-2) {$b_{m-2}$};
\node[fill=none, draw=none, inner sep=0pt, right=5pt] at (b3m-1) {$b_{m-1}$};
\node[fill=none, draw=none, inner sep=0pt, right=5pt] at (b3m) {$b_m$};
\node[fill=none, draw=none, inner sep=0pt, right=5pt] at (b3n-1) {$b_{n-1}$};
\node[fill=none, draw=none, inner sep=0pt, right=5pt] at (b3n) {$b_{n}$};

\draw (a11) .. controls (-0.1,2.5) and (-0.1,2.3) .. (a12);
\draw (a1m-1) .. controls (-0.1,-1.5) and (-0.1,-1.7) .. (a1m);
\draw (a1m-3) .. controls (-0.1,0.1) and (-0.1,-0.1) .. (a1m-2);
\draw (a11)--(b11);
\draw (a12)--(b12);
\draw (a1m-3)--(b1m-3);
\draw (a1m-2)--(b1m-2);
\draw (a1m-1)--(b1m-1);

\draw (a21) .. controls (4.4,2.5) and (4.4,2.3) .. (a22);
\draw (a2m-1) .. controls (4.4,-1.5) and (4.4,-1.7) .. (a2m);
\draw (a2m-3) .. controls (4.4,0.1) and (4.4,-0.1) .. (a2m-2);
\draw (a21)--(b21);
\draw (a22)--(b22);
\draw (a2m-3)--(b2m-3);
\draw (a2m-2)--(b2m-2);
\draw (a2m-1)--(b2m-1);
\draw (a2m)--(b2m);

\draw (a31) .. controls (8.9,2.5) and (8.9,2.3) .. (a32);
\draw (a3m-1) .. controls (8.9,-1.5) and (8.9,-1.7) .. (a3m);
\draw (a3m-3) .. controls (8.9,0.1) and (8.9,-0.1) .. (a3m-2);
\draw (a31)--(b31);
\draw (a32)--(b32);
\draw (a3m-3)--(b3m-3);
\draw (a3m-2)--(b3m-2);
\draw (a3m-1)--(b3m-1);
\draw (a3m)--(b3m);
\draw (vdots7)--(a3m)--(b3n-1);

\end{tikzpicture}
    \caption{The dominating sets (\ref{dom4m1}), (\ref{dom4m2}), (\ref{dom4m3}) respectively, for $m\geq 4$. Each edge represents a vertex in $\mathcal{X}$.}
    \label{fig:dom4m}
\end{figure}
\newpage
Next, we consider the case (C2).

\begin{lemma}
\label{dom3nupper}
    For $n\geq 3$ we have $\gamma(F_2(K_{3,n}))\leq n$.
\end{lemma}

\begin{proof}
    The set
    \begin{equation}
    \label{dom3n}
        \bigcup_{i=1}^{n-2}\{\{a_1,b_i\}\}\cup \{\{a_2,a_3\},\{b_{n-1},b_{n}\}\}.
    \end{equation}
    is a dominating set of $F_2(K_{3,n})$ that has size $n$. The result follows.
\end{proof}

We illustrate the set (\ref{dom3n}) in Figure \ref{fig:dom3n}.

\begin{figure}[!ht]
    \centering
        \begin{tikzpicture}[
    every node/.style={circle, fill=black, inner sep=2.5pt},
    >=stealth
]

\node (a1) at (0,2.8) {};
\node (a2) at (0,2.0) {};
\node (a3) at (0,1.2) {};

\node (b1) at (1.8,2.8) {};
\node[draw=none, fill=none, shape=rectangle] (vdots) at (1.8, 2.0) {$\vdots$};
\node (bn-2) at (1.8,1.2) {};
\node (bn-1) at (1.8,0.4) {};
\node (bn) at (1.8,-0.4) {};

\node[fill=none, draw=none, inner sep=0pt, left=5pt]  at (a1) {$a_1$};
\node[fill=none, draw=none, inner sep=0pt, left=5pt]  at (a2) {$a_2$};
\node[fill=none, draw=none, inner sep=0pt, left=5pt]  at (a3) {$a_3$};

\node[fill=none, draw=none, inner sep=0pt, right=5pt] at (b1) {$b_1$};
\node[fill=none, draw=none, inner sep=0pt, right=5pt] at (bn-2) {$b_{n-2}$};
\node[fill=none, draw=none, inner sep=0pt, right=5pt] at (bn-1) {$b_{n-1}$};
\node[fill=none, draw=none, inner sep=0pt, right=5pt] at (bn) {$b_n$};

\draw (a2) .. controls (-0.1,1.7) and (-0.1,1.5) .. (a3);
\draw (bn-1) .. controls (1.9,0.1) and (1.9,-0.1) .. (bn);
\draw (b1)--(a1)--(bn-2);
\draw (a1)--(vdots);

\end{tikzpicture}
    \caption{The dominating set (\ref{dom3n}).}
    \label{fig:dom3n}
\end{figure}
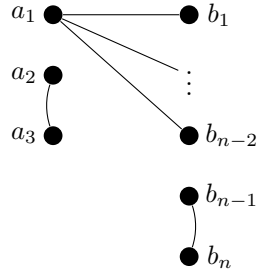

Next, we consider the case (C3).

\begin{lemma}
\label{c3upper}
    For odd $m\geq 5$, we have 
    \begin{equation*}
        \gamma(F_2(K_{m,m+1}))\leq \frac{3m-1}{2}.
    \end{equation*}
\end{lemma}

\begin{proof}
    The set
    \begin{equation}
    \label{dommm+1}
        \bigcup_{i=1}^{m-1} \{\{a_i,b_i\}\}\cup \bigcup_{j=1}^{(m+1)/2}\{\{b_{2j-1},b_{2j}\}\}
    \end{equation}
    is a dominating set of $F_2(K_{m,m+1})$ that has size $\dfrac{3m-1}{2}$. The result follows.
\end{proof}

We illustrate the set (\ref{dommm+1}) in Figure \ref{fig:dommm+1}.

\begin{figure}[!ht]
    \centering
        \begin{tikzpicture}[
    every node/.style={circle, fill=black, inner sep=2.5pt},
    >=stealth
]

\node (a1) at (0,2.8) {};
\node (a2) at (0,2.0) {};
\node (a3) at (0,1.2) {};
\node (a4) at (0,0.4) {};
\node (a5) at (0,-0.4) {};

\node (b1) at (1.8,2.8) {};
\node (b2) at (1.8, 2.0) {};
\node (b3) at (1.8,1.2) {};
\node (b4) at (1.8,0.4) {};
\node (b5) at (1.8,-0.4) {};
\node (b6) at (1.8,-1.2) {};

\node[fill=none, draw=none, inner sep=0pt, left=5pt]  at (a1) {$a_1$};
\node[fill=none, draw=none, inner sep=0pt, left=5pt]  at (a2) {$a_2$};
\node[fill=none, draw=none, inner sep=0pt, left=5pt]  at (a3) {$a_3$};
\node[fill=none, draw=none, inner sep=0pt, left=5pt]  at (a4) {$a_4$};
\node[fill=none, draw=none, inner sep=0pt, left=5pt]  at (a5) {$a_5$};

\node[fill=none, draw=none, inner sep=0pt, right=5pt] at (b1) {$b_1$};
\node[fill=none, draw=none, inner sep=0pt, right=5pt] at (b2) {$b_2$};
\node[fill=none, draw=none, inner sep=0pt, right=5pt] at (b3) {$b_3$};
\node[fill=none, draw=none, inner sep=0pt, right=5pt] at (b4) {$b_4$};
\node[fill=none, draw=none, inner sep=0pt, right=5pt] at (b5) {$b_5$};
\node[fill=none, draw=none, inner sep=0pt, right=5pt] at (b6) {$b_6$};

\draw (b1) .. controls (1.9,2.5) and (1.9,2.3) .. (b2);
\draw (b3) .. controls (1.9,0.9) and (1.9,0.7) .. (b4);
\draw (b5) .. controls (1.9,-0.7) and (1.9,-0.9) .. (b6);
\draw (a1)--(b1);
\draw (a2)--(b2);
\draw (a3)--(b3);
\draw (a4)--(b4);

\end{tikzpicture}
    \caption{The dominating set (\ref{dommm+1}) for the case $m=5$.}
    \label{fig:dommm+1}
\end{figure}
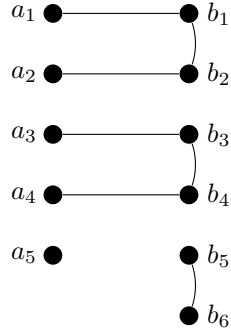

Next, we consider the case (C4).

\begin{lemma}
\label{5mn}
    For odd $m\geq 5$ and a positive integer $n\geq m$ such that $n\neq m+1$, we have
    \begin{equation*}
        \gamma(F_2(K_{m,n}))\leq \frac{m-1}{2}+n.
    \end{equation*}
\end{lemma}

\begin{proof}
    Assume that $n=m$. The set 
    \begin{equation}
    \label{dom5n1}
        \bigcup_{i=1}^{(m-1)/2}\{\{a_{2i-1},a_{2i}\}\} \cup \bigcup_{j=1}^{m-1} \{\{a_j,b_j\}\}\cup \{\{a_{m-1},a_{m}\}\}
    \end{equation}
    is a dominating set of $F_2(K_{m,m})$ that has size $\dfrac{3m-1}{2}$. 

    Assume that $n\geq m+2$. The set
    \begin{equation}
    \label{dom5n2}
        \bigcup_{j=1}^{(m-1)/2}\{\{a_{2j-1},a_{2j}\}\}\cup \bigcup_{h=1}^{m-1} \{\{a_h,b_h\}\}\cup \bigcup_{i=m}^{n-1} \{\{a_m,b_i\}\}\cup\{\{a_{m-1},a_{m}\}\}
    \end{equation}
    is a dominating set of $F_2(K_{m,n})$ that has size $\dfrac{m-1}{2}+n$. The result follows.
\end{proof}

We illustrate the sets (\ref{dom5n1}), (\ref{dom5n2}) in Figure \ref{fig:dom5}.

\begin{figure}[!h]
    \centering
        \begin{tikzpicture}[
    every node/.style={circle, fill=black, inner sep=2.5pt},
    >=stealth
]

\node (a11) at (0,2.8) {};
\node (a12) at (0,2.0) {};
\node[draw=none, fill=none, shape=rectangle] (vdots1) at (0, 1.2) {$\vdots$};
\node (a1m-4) at (0,0.4) {};
\node (a1m-3) at (0,-0.4) {};
\node (a1m-2) at (0,-1.2) {};
\node (a1m-1) at (0,-2) {};
\node (a1m) at (0,-2.8) {};

\node (a21) at (4.5,2.8) {};
\node (a22) at (4.5,2.0) {};
\node[draw=none, fill=none, shape=rectangle] (vdots2) at (4.5, 1.2) {$\vdots$};
\node (a2m-4) at (4.5,0.4) {};
\node (a2m-3) at (4.5,-0.4) {};
\node (a2m-2) at (4.5,-1.2) {};
\node (a2m-1) at (4.5,-2) {};
\node (a2m) at (4.5,-2.8) {};

\node (b11) at (1.8,2.8) {};
\node (b12) at (1.8, 2.0) {};
\node[draw=none, fill=none, shape=rectangle] (vdots3) at (1.8, 1.2) {$\vdots$};
\node (b1m-4) at (1.8,0.4) {};
\node (b1m-3) at (1.8,-0.4) {};
\node (b1m-2) at (1.8,-1.2) {};
\node (b1m-1) at (1.8,-2) {};
\node (b1m) at (1.8,-2.8) {};

\node (b21) at (6.3,2.8) {};
\node (b22) at (6.3, 2.0) {};
\node[draw=none, fill=none, shape=rectangle] (vdots4) at (6.3, 1.2) {$\vdots$};
\node (b2m-4) at (6.3,0.4) {};
\node (b2m-3) at (6.3,-0.4) {};
\node (b2m-2) at (6.3,-1.2) {};
\node (b2m-1) at (6.3,-2) {};
\node (b2m) at (6.3,-2.8) {};
\node[draw=none, fill=none, shape=rectangle] (vdots5) at (6.3, -3.6) {$\vdots$};
\node (b2n-1) at (6.3,-4.4) {};
\node (b2n) at (6.3,-5.2) {};

\node[fill=none, draw=none, inner sep=0pt, left=5pt]  at (a11) {$a_1$};
\node[fill=none, draw=none, inner sep=0pt, left=5pt]  at (a12) {$a_2$};
\node[fill=none, draw=none, inner sep=0pt, left=5pt]  at (a1m-4) {$a_{m-4}$};
\node[fill=none, draw=none, inner sep=0pt, left=5pt]  at (a1m-3) {$a_{m-3}$};
\node[fill=none, draw=none, inner sep=0pt, left=5pt]  at (a1m-2) {$a_{m-2}$};
\node[fill=none, draw=none, inner sep=0pt, left=5pt]  at (a1m-1) {$a_{m-1}$};
\node[fill=none, draw=none, inner sep=0pt, left=5pt]  at (a1m) {$a_m$};

\node[fill=none, draw=none, inner sep=0pt, left=5pt]  at (a21) {$a_1$};
\node[fill=none, draw=none, inner sep=0pt, left=5pt]  at (a22) {$a_2$};
\node[fill=none, draw=none, inner sep=0pt, left=5pt]  at (a2m-4) {$a_{m-4}$};
\node[fill=none, draw=none, inner sep=0pt, left=5pt]  at (a2m-3) {$a_{m-3}$};
\node[fill=none, draw=none, inner sep=0pt, left=5pt]  at (a2m-2) {$a_{m-2}$};
\node[fill=none, draw=none, inner sep=0pt, left=5pt]  at (a2m-1) {$a_{m-1}$};
\node[fill=none, draw=none, inner sep=0pt, left=5pt]  at (a2m) {$a_m$};

\node[fill=none, draw=none, inner sep=0pt, right=5pt] at (b11) {$b_1$};
\node[fill=none, draw=none, inner sep=0pt, right=5pt] at (b12) {$b_2$};
\node[fill=none, draw=none, inner sep=0pt, right=5pt] at (b1m-4) {$b_{m-4}$};
\node[fill=none, draw=none, inner sep=0pt, right=5pt] at (b1m-3) {$b_{m-3}$};
\node[fill=none, draw=none, inner sep=0pt, right=5pt] at (b1m-2) {$b_{m-2}$};
\node[fill=none, draw=none, inner sep=0pt, right=5pt] at (b1m-1) {$b_{m-1}$};
\node[fill=none, draw=none, inner sep=0pt, right=5pt] at (b1m) {$b_m$};

\node[fill=none, draw=none, inner sep=0pt, right=5pt] at (b21) {$b_1$};
\node[fill=none, draw=none, inner sep=0pt, right=5pt] at (b22) {$b_2$};
\node[fill=none, draw=none, inner sep=0pt, right=5pt] at (b2m-4) {$b_{m-4}$};
\node[fill=none, draw=none, inner sep=0pt, right=5pt] at (b2m-3) {$b_{m-3}$};
\node[fill=none, draw=none, inner sep=0pt, right=5pt] at (b2m-2) {$b_{m-2}$};
\node[fill=none, draw=none, inner sep=0pt, right=5pt] at (b2m-1) {$b_{m-1}$};
\node[fill=none, draw=none, inner sep=0pt, right=5pt] at (b2m) {$b_m$};
\node[fill=none, draw=none, inner sep=0pt, right=5pt] at (b2n-1) {$b_{n-1}$};
\node[fill=none, draw=none, inner sep=0pt, right=5pt] at (b2n) {$b_n$};

\draw (a11) .. controls (-0.1,2.5) and (-0.1,2.3) .. (a12);
\draw (a1m-2) .. controls (-0.1,-1.5) and (-0.1,-1.7) .. (a1m-1);
\draw (a1m-4) .. controls (-0.1,0.1) and (-0.1,-0.1) .. (a1m-3);
\draw (a1m-1) .. controls (-0.1,-2.3) and (-0.1,-2.5) .. (a1m);
\draw (a11)--(b11);
\draw (a12)--(b12);
\draw (a1m-4)--(b1m-4);
\draw (a1m-3)--(b1m-3);
\draw (a1m-2)--(b1m-2);
\draw (a1m-1)--(b1m-1);

\draw (a21) .. controls (4.4,2.5) and (4.4,2.3) .. (a22);
\draw (a2m-2) .. controls (4.4,-1.5) and (4.4,-1.7) .. (a2m-1);
\draw (a2m-4) .. controls (4.4,0.1) and (4.4,-0.1) .. (a2m-3);
\draw (a2m-1) .. controls (4.4,-2.3) and (4.4,-2.5) .. (a2m);
\draw (a21)--(b21);
\draw (a22)--(b22);
\draw (a2m-4)--(b2m-4);
\draw (a2m-3)--(b2m-3);
\draw (a2m-2)--(b2m-2);
\draw (a2m-1)--(b2m-1);
\draw (a2m)--(b2m);
\draw (a2m)--(vdots5);
\draw (a2m)--(b2n-1);

\end{tikzpicture}
    \caption{The dominating sets (\ref{dom5n1}), (\ref{dom5n2}) respectively.}
    \label{fig:dom5}
\end{figure}
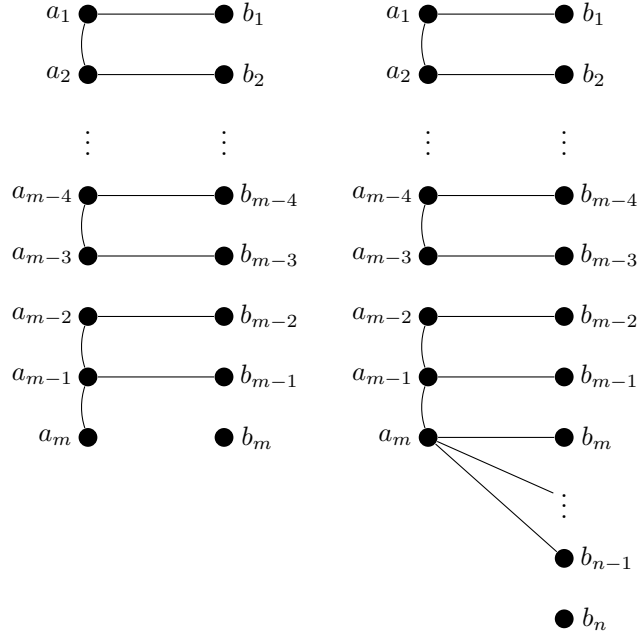
\newpage

\section{A lower bound of $\gamma(F_2(K_{m,n}))$}
\label{F2}
We continue to work with $G=K_{m,n}$ where $2\leq m\leq n$. In this section we present a lower bound on the domination number of $F_2(K_{m,n})$.

\begin{definition}
    Let $2\leq m\leq n$. For $1\leq i\leq m$ and a subset $S\subseteq \mathcal{X}$, define
    \begin{align}
        &\sigma_i(S)=\vert \{b\in B\mid \{a_i,b\}\in S\}\vert, &s(S)=\vert \{i\mid \sigma_i(S)=0\}\vert,\label{sigmadef}\\
        &\pi_i(S)=\vert \{a\in A\mid \{a_i,a\}\in S\}\vert, &p(S)=m-s(S)-\vert \{i\mid \pi_i(S)>0\}\vert.\label{pidef}
    \end{align}

    For $1\leq j\leq n$ and a subset $S\subseteq \mathcal{X}$, define
    \begin{align}
        &\tau_j(S)=\vert \{a\in A\mid \{a,b_j\}\in S\}\vert, &t(S)=\vert \{j\mid \tau_j(S)=0\}\vert,\label{taudef}\\
        &\rho_j(S)=\vert \{b\in B\mid \{b,b_j\}\in S\}\vert, &r(S)=n-t(S)-\vert \{j\mid \rho_j(S)>0\}\vert.\label{rhodef}
    \end{align}

    If the context is clear, we abbreviate $\sigma_i=\sigma_i(S)$, $s=s(S)$, $\pi_i=\pi_i(S)$, $p=p(S)$, $\tau_j=\tau_j(S)$, $t=t(S)$, $\rho_j=\rho_j(S)$, $r=r(S)$.
\end{definition}

Next we present some useful lemmas.

\begin{lemma}
\label{important}
    Let $2\leq m\leq n$. Let $S$ denote a dominating set of $F_2(K_{m,n})$. For each ordered pair $(i,j)$ ($1\leq i\leq m$, $1\leq j\leq n$), at least one of the following (i)--(iii) hold:
    \begin{enumerate}[label=(\roman*)]
        \item $\{a_i,b_j\}\in S$;
        \item $\pi_i>0$;
        \item $\rho_j>0$.
    \end{enumerate}
\end{lemma}
\begin{proof}
    Pick a pair $i,j$ that does not satisfy (i). Then $\{a_i,b_j\}\in N[S]\setminus S$. Since $\{a_i,b_j\}$ is adjacent to some vertex in $\mathcal{C}_0\cup \mathcal{C}_2$, at least one of (ii), (iii) holds. The result follows.
\end{proof}

\begin{lemma}
    \label{pr<2}
    Let $2\leq m\leq n$. Let $S$ denote a dominating set of $F_2(K_{m,n})$. Then there exists a dominating set $S'$ of $F_2(K_{m,n})$ such that $\vert S'\vert\leq \vert S\vert$ and satisfies at least one of the following (i), (ii):
    \begin{enumerate}[label=(\roman*)]
        \item $p(S')\leq 1$;
        \item $r(S')\leq 1$.
    \end{enumerate}
\end{lemma}

\begin{proof}
    If $S$ satisfies either (i) or (ii), then set $S'=S$. Assume otherwise. Without loss, assume that $\pi_1(S)=\pi_2(S)=\rho_1(S)=\rho_2(S)=0$. By Lemma \ref{important}, we have
    \begin{equation*}
        \{\{a_1,b_1\}, \{a_1,b_2\}, \{a_2,b_1\}, \{a_2,b_2\}\}\subseteq S.
    \end{equation*}
    Set
    \begin{equation*}
        T=S\setminus\{\{a_1,b_2\},\{a_2,b_1\}\}\cup \{a_1,b_1\}.
    \end{equation*}
    Observe that $\vert T\vert<\vert S\vert$, $p(T)=p(S)-2$, and $N[T]\supseteq N[S]$. Iterate the process above until $p(T)\leq 1$, and set $S'=T$. The result follows. 
\end{proof}

In view of Lemma \ref{pr<2}, from now on, we only consider dominating sets of $F_2(K_{m,n})$ that have $p\leq 1$ or $r\leq 1$.

\begin{lemma}
\label{st<3}
    Let $2\leq m\leq n$. Let $S$ denote a dominating set of $F_2(K_{m,n})$. Then there exists a dominating set $S'$ of $F_2(K_{m,n})$ that satisfies the following (i)--(iii):
    \begin{enumerate}[label=(\roman*)]
        \item $s(S')\leq 2$;
        \item $t(S')\leq 2$;
        \item $\vert S'\vert \leq \vert S\vert$.
    \end{enumerate}
\end{lemma}

\begin{proof}
    To avoid triviality, assume that $3\leq n$. We first show that there exists $S'$ that satisfies (ii) and (iii). If $S$ satisfies (ii) then set $S'=S$. Assume otherwise. Pick $h,j,\ell$ such that $\tau_{h}(S)=\tau_{j}(S)=\tau_{\ell}(S)=0$. Observe that
    \begin{equation*}
        \{\{b_h,b_j\}, \{b_h,b_\ell\}, \{b_j,b_\ell\}\}\subseteq S.
    \end{equation*}
    Set 
    \begin{equation*}
        T=S\setminus \{\{b_h,b_j\}\}\cup \{\{a_1,b_h\}\}.
    \end{equation*}
    Observe that 
    \begin{equation*}
        s(T)\leq s(S), \qquad \qquad t(T)=t(S)-1, \qquad \qquad N[T]\supseteq N[S], \qquad \qquad \vert T\vert \leq \vert S\vert.
    \end{equation*}

    Iterate the process above until $t(T)\leq 2$, and set $S'=T$. Then $S'$ satisfies (ii) and (iii). The result follows by symmetry.
\end{proof}

In view of Lemma \ref{st<3}, from now on, we only consider dominating sets of $F_2(K_{m,n})$ that have $s,t\leq 2$.

\begin{lemma}
    \label{C2lem}
    Let $2\leq m\leq n$. Let $S$ denote a dominating set of $F_2(K_{m,n})$. Then 
    \begin{equation}
        \label{C2ineq}
        \vert S\cap \mathcal{C}_2\vert\geq \left\lfloor \frac{s}{2}\right\rfloor+\left\lceil \frac{m-p-s}{2}\right\rceil. 
    \end{equation}
\end{lemma}

\begin{proof}
    Assume that $s=2$. Without loss, assume that $\sigma_{m-1}=\sigma_m=0$. If $\{a_{m-1},a_{m}\}\not\in S$, then observe that $\{a_{m-1},a_m\}\not\in N[S]$, which is a contradiction. Hence, $\{a_{m-1},a_{m}\}\in S$. The result follows from this fact and routine combinatorial counting. 
\end{proof}

\begin{lemma}
    \label{C1lem}
    Let $2\leq m\leq n$. Let $S$ denote a dominating set of $F_2(K_{m,n})$. Then 
    \begin{equation}
        \label{C1ineq}
        \vert S\cap \mathcal{C}_1\vert\geq \max\{m+p(r-1)-s,n+r(p-1)-t,m-s,n-t\}.
    \end{equation}
\end{lemma}

\begin{proof}
    It is clear from the definitions of $s,t$ that
    \begin{equation*}
        \vert S\cap \mathcal{C}_1\vert\geq \max\{m-s,n-t\}.
    \end{equation*}
    It suffices to show that
    \begin{equation*}
        \vert S\cap \mathcal{C}_1\vert \geq \max\{m+p(r-1)-s,n+r(p-1)-t\}.
    \end{equation*}
    Consider a pair of elements $a_i\in A, b_j\in B$ such that $\sigma_i>0$, $\pi_i=0$, $\tau_j>0$, $\rho_j=0$. By Lemma \ref{important}, we have $\{a_{i},b_{j}\}\in S$. In addition, observe that at least $m-p-s$ elements in $A$ are contained in some vertex in $S\cap \mathcal{C}_1$, and at least $n-r-t$ elements in $B$ are contained in some vertex in $S\cap \mathcal{C}_1$. By combinatorial counting,
    \begin{equation*}
        \vert S\cap \mathcal{C}_1\vert \geq \max \{m-p-s,n-r-t\}+ps=\max\{m+p(r-1)-s,n+r(p-1)-t\}.
    \end{equation*}
    The result follows.
\end{proof}

\begin{lemma}
    \label{C0lem}
    Let $2\leq m\leq n$. Let $S$ denote a dominating set of $F_2(K_{m,n})$. Then 
    \begin{equation}
        \label{C0ineq}
        \vert S\cap \mathcal{C}_0\vert\geq \left\lfloor \frac{t}{2}\right\rfloor+\left\lceil\frac{n-r-t}{2}\right\rceil. 
    \end{equation}
\end{lemma}

\begin{proof}
    Similar to (\ref{C2ineq}).
\end{proof}

\begin{lemma}
\label{biglemma}
    Let $2\leq m\leq n$. Let $S$ denote a dominating set of $F_2(K_{m,n})$. Then the following (\ref{bigineq1})--(\ref{bigineq4}) hold:
    \begin{align}
        \vert S\vert&\geq\left\lfloor \frac{s}{2}\right\rfloor+\left\lceil\frac{m-p-s}{2}\right\rceil+m+p(r-1)-s+\left\lfloor\frac{t}{2}\right\rfloor+\left\lceil\frac{n-r-t}{2}\right\rceil.\label{bigineq1}\\
        \vert S\vert&\geq\left\lfloor \frac{s}{2}\right\rfloor+\left\lceil\frac{m-p-s}{2}\right\rceil+n+r(p-1)-t+\left\lfloor\frac{t}{2}\right\rfloor+\left\lceil\frac{n-r-t}{2}\right\rceil.\label{bigineq2}\\
        \vert S\vert&\geq\left\lfloor \frac{s}{2}\right\rfloor+\left\lceil\frac{m-p-s}{2}\right\rceil+m-s+\left\lfloor\frac{t}{2}\right\rfloor+\left\lceil\frac{n-r-t}{2}\right\rceil.\label{bigineq3}\\
        \vert S\vert&\geq\left\lfloor \frac{s}{2}\right\rfloor+\left\lceil\frac{m-p-s}{2}\right\rceil+n-t+\left\lfloor\frac{t}{2}\right\rfloor+\left\lceil\frac{n-r-t}{2}\right\rceil.\label{bigineq4}
    \end{align}
\end{lemma}

\begin{proof}
    Combine (\ref{C2ineq})--(\ref{C0ineq}).
\end{proof}



We now consider the case (C1).

\begin{lemma}
    \label{yesC2}
    Let $n\geq 2$. For every dominating set $S$ of $F_2(K_{2,n})$, there exists a dominating set $S'$ of $F_2(K_{2,n})$ such that $ \mathcal{C}_2\subseteq S'$ and $\vert S'\vert\leq \vert S\vert$.
\end{lemma}
\begin{proof}
    If $\mathcal{C}_2\subseteq S$, then set $S'=S$. Assume otherwise. By Lemma \ref{importantxy}, the set $S\cap \mathcal{C}_1$ is not empty. Without loss, assume that $\{a_1,b_1\}\in S$. 

    By Lemma \ref{important}, either (i) $\{a_2,b_1\}\in S$, or (ii) $\rho_1(S)>0$. In the case (i), set 
    \begin{equation*}
        S'=S\setminus \{\{a_1,b_1\}\}\cup \{\{a_1,a_2\}\}.
    \end{equation*}
    If $S$ satisfies case (ii) and not case (i), let $\{b_1,b\}\in S$ for some $b\in B\setminus \{b_1\}$. Set
    \begin{equation*}
        S'=S\setminus \{\{b_1,b\}\}\cup \{\{a_1,a_2\}\}.
    \end{equation*}
    It is routine to verify that $S'$ is a dominating set that satisfies the result of the lemma.
\end{proof}

\begin{lemma}
\label{noC0}
    Let $n\geq 2$. For every dominating set $S$ of $F_2(K_{2,n})$, there exists a dominating set $S'$ of $F_2(K_{2,n})$ such that $S'\cap \mathcal{C}_0=\emptyset$ and $\vert S'\vert\leq \vert S\vert$.
\end{lemma}
\begin{proof}
    In view of Lemma \ref{yesC2}, assume without loss that $\{a_1,a_2\}\in S$. Let
    \begin{equation*}
        S'=(S\cap (\mathcal{C}_1\cup \mathcal{C}_2))\cup \bigcup_{\substack{1\leq i<j\leq n\\\{b_i,b_j\}\in S}}\{\{a_1,b_i\}\}.
    \end{equation*}
    It is routine to verify that $S'$ satisfies the result of the lemma.  
\end{proof}

\begin{lemma}
\label{evenlower}
    For even $m$ and $n\geq m$, we have
    \begin{equation*}
        \gamma(F_2(K_{m,n}))\geq \frac{m}{2}+n-1.
    \end{equation*}
\end{lemma}

\begin{proof}
    Assume that $m=2$. Pick a dominating set $S$; in view of Lemmas \ref{yesC2}, \ref{noC0} we assume without loss that $\mathcal{C}_2\subseteq S$ and $S\cap \mathcal{C}_0=\emptyset$. By Lemma \ref{importantxy}, it is necessary that $t\leq 1$. Hence, $\vert S\cap \mathcal{C}_1\vert \geq n-1$. The result follows.
    
    We now assume that $m\geq 4$. Pick a dominating set $S$. In view of Lemma \ref{pr<2}, we consider the following cases: (i) $p=0$; (ii) $p=1$; (iii) $p>1$ and $r=0$; (iv) $p>1$ and $r=1$. 
    
    Assume that $p=0$. By (\ref{bigineq4}),
    \begin{equation*}
        \vert S\vert \geq\left\lfloor \frac{s}{2}\right\rfloor+\left\lceil\frac{m-s}{2}\right\rceil+n-t+\left\lfloor\frac{t}{2}\right\rfloor+\left\lceil\frac{n-r-t}{2}\right\rceil\geq \frac{m}{2}+n-1.
    \end{equation*}

    Assume that $p=1$. By (\ref{bigineq4}),
    \begin{equation}
    \label{p=1ineq}
        \vert S\vert\geq\left\lfloor \frac{s}{2}\right\rfloor+\left\lceil\frac{m-1-s}{2}\right\rceil+n-t+\left\lfloor\frac{t}{2}\right\rfloor+\left\lceil\frac{n-r-t}{2}\right\rceil \geq \frac{m}{2}+n-2.
    \end{equation}
    It is easy to verify that the equality cases for the last inequality of (\ref{p=1ineq}) are $(r,s,t)=(n-1,1,1), (n-2,1,2)$. Pick and assume an equality case. Since $s=1$, assume without loss that $\sigma_1=0$. Since $p=1$, we have $\pi_1=0$. Since $r\geq n-2\geq 2$, there exists $j$ $(1\leq j\leq n)$ such that the ordered pair $(1,j)$ violates Lemma \ref{important}. Hence, the last inequality of (\ref{p=1ineq}) is strict, and the result follows. 

    Assume that $p>1$ and $r=0$. There are two cases to consider: $m=n$, and $m\leq n-1$. Assume that $m=n$. By (\ref{bigineq3}),
    \begin{equation*}
        \vert S\vert\geq\left\lfloor \frac{s}{2}\right\rfloor+\left\lceil\frac{m-p-s}{2}\right\rceil+m-s+\left\lfloor\frac{t}{2}\right\rfloor+\left\lceil\frac{n-t}{2}\right\rceil\geq \frac{m}{2}+m-1.
    \end{equation*}
    The result follows. Now assume that $m\leq n-1$. By (\ref{bigineq4}),
    \begin{equation*}
        \vert S\vert\geq\left\lfloor \frac{s}{2}\right\rfloor+\left\lceil\frac{m-p-s}{2}\right\rceil+n-t+\left\lfloor\frac{t}{2}\right\rfloor+\left\lceil\frac{n-t}{2}\right\rceil\geq \frac{m}{2}+n-1.
    \end{equation*}

    Assume that $p>1$ and $r=1$. By (\ref{bigineq2}) and the assumption $4\leq m\leq n$,
    \begin{equation*}
        \vert S\vert\geq\left\lfloor \frac{s}{2}\right\rfloor+\left\lceil\frac{m+p-s}{2}\right\rceil+n-1-t+\left\lfloor\frac{t}{2}\right\rfloor+\left\lceil\frac{n-1-t}{2}\right\rceil\geq \frac{m}{2}+n-1.
    \end{equation*}
    The result follows.
\end{proof}

Next we consider the case (C2). 

\begin{lemma}
\label{dom3nlower}
    For $n\geq 3$ we have
    \begin{equation*}
        \gamma(F_2(K_{3,n}))\geq n.
    \end{equation*}
\end{lemma}

\begin{proof}
    Pick a dominating set $S$. In view of Lemma \ref{pr<2}, we consider the following cases: (i) $p=0$; (ii) $p=1$; (iii) $p>1$ and $r=0$; (iv) $p>1$ and $r=1$. 

    For the case $p=0$, the proof is similar to that of Lemma \ref{evenlower} when $m\geq 4$, hence omitted.

    Assume that $p=1$. By (\ref{bigineq4}),
    \begin{equation*}
        \vert S\vert\geq\left\lfloor \frac{s}{2}\right\rfloor+\left\lceil\frac{2-s}{2}\right\rceil+n-t+\left\lfloor\frac{t}{2}\right\rfloor+\left\lceil\frac{n-r-t}{2}\right\rceil\geq n.
    \end{equation*}

    For the case $p>1$ and $r=0$, the proof is similar to that of Lemma \ref{evenlower} when $4\leq m\leq n-1$, hence omitted.

    For the case $p>1$ and $r=1$, the proof is similar to that of Lemma \ref{evenlower} when $m\geq 4$, hence omitted.
\end{proof}

Next we consider the case (C3).

\begin{lemma}
    For odd $m\geq 5$, we have
    \begin{equation*}
        \gamma(F_2(K_{m,m+1}))\geq \frac{3m-1}{2}.
    \end{equation*}
\end{lemma}

\begin{proof}
    Pick a dominating set $S$. In view of Lemma \ref{pr<2}, we consider the following cases: (i) $p=0$; (ii) $p=1$; (iii) $p>1$ and $r=0$; (iv) $p>1$ and $r=1$. 

    For the case $p=0$, the proof is similar to that of Lemma \ref{evenlower} when $m\geq 4$, hence omitted.

    For the case $p=1$, the proof is similar to that of Lemma \ref{dom3nlower}, hence omitted.

    For the case $p>1$ and $r=0$, the proof is similar to that of Lemma \ref{evenlower} when $4\leq m\leq n-1$, hence omitted.

    For the case $p>1$ and $r=1$, the proof is similar to that of Lemma \ref{evenlower} when $m\geq 4$, hence omitted.
\end{proof}

Next we consider the case (C4).

\begin{lemma}
\label{5mnlower}
    For odd $m\geq 5$ and a positive integer $n\geq m$ such that $n\neq m+1$, we have
    \begin{equation*}
        \gamma(F_2(K_{m,n}))\geq \frac{m-1}{2}+n.
    \end{equation*}
\end{lemma}

\begin{proof}
    Pick a dominating set $S$. In view of Lemma \ref{pr<2}, we consider the following cases: (i) $p=0$; (ii) $p=1$; (iii) $p>1$ and $r=0$; (iv) $p>1$ and $r=1$.

    Assume that $p=0$. By (\ref{bigineq4}),
    \begin{equation}
    \label{p=0ineq}
        \vert S\vert \geq\left\lfloor \frac{s}{2}\right\rfloor+\left\lceil\frac{m-s}{2}\right\rceil+n-t+\left\lfloor\frac{t}{2}\right\rfloor+\left\lceil\frac{n-r-t}{2}\right\rceil\geq \frac{m-1}{2}+n-1.
    \end{equation}
    It is easy to verify that the equality cases for the last inequality of (\ref{p=0ineq}) are $(r,s,t)=(n-1,1,1), (n-2,1,2)$. Pick and assume an equality case. Since $s=1$, assume without loss that $\sigma_1=0$. Since $p=0$, we have $\pi_1=0$. Since $r\geq n-2\geq 3$, there exists $j$ ($1\leq j\leq n$) such that the ordered pair $(1,j)$ violates Lemma \ref{important}. Hence, the last inequality of (\ref{p=0ineq}) is strict, and the result follows. 

    Assume that $p=1$. Assume without loss that $\pi_1=0$. By (\ref{bigineq4}),
    \begin{equation}
    \label{p=1ineq2}
        \vert S\vert\geq\left\lfloor \frac{s}{2}\right\rfloor+\left\lceil\frac{m-1-s}{2}\right\rceil+n-t+\left\lfloor\frac{t}{2}\right\rfloor+\left\lceil\frac{n-r-t}{2}\right\rceil \geq \frac{m-1}{2}+n-1.
    \end{equation}
    It is easy to verify that the equality cases for the last inequality of (\ref{p=1ineq2}) are $(r,t)=(n-1,1), (n-2,2)$. Assume that $(r,t)=(n-1,1)$. Assume without loss that $\tau_1=0$. Then the ordered pair $(1,1)$ violates Lemma \ref{important}. Now assume that $(r,t)=(n-2,2)$. In this case, observe that $\vert S\cap \mathcal{C}_1\vert=n-2$. Assume without loss that $\rho_j=0$ for $1\leq j\leq n-2$. By Lemma \ref{important}, we have $\sigma_1\geq n-2$. Hence,
    \begin{equation*}
        \vert S\cap \mathcal{C}_1\vert\geq \sum_{i=1}^{m}\sigma_i\geq n-2+(m-s-1)\geq n,
    \end{equation*}
    which is a contradiction. Hence, the last inequality of (\ref{p=1ineq2}) is strict, and the result follows.

    For the case $p>1$ and $r=0$, the proof is similar to that of Lemma \ref{evenlower} when $4\leq m\leq n-1$, hence omitted.

    For the case $p>1$ and $r=1$, the proof is similar to that of Lemma \ref{evenlower} when $m\geq 4$, hence omitted.
\end{proof}

We now prove Theorem \ref{main1}.

\begin{proof}[Proof of Theorem \ref{main1}]
    Combine Lemmas \ref{evenupper}--\ref{5mn}, \ref{evenlower}--\ref{5mnlower}.
\end{proof}

\section{The domination number of $F_k(K_{m,n})$ for $k\geq 3$}
We continue to work with $G=K_{m,n}$, where $2\leq m\leq n$. In this section we prove Theorem \ref{main2}. 

Our first goal is to prove the left inequality of (\ref{intro2}).

\begin{lemma}
    Recall the notation $\Delta$ from Lemma \ref{DeltaG}. For positive integers $m,n,k$ such that $2\leq m\leq n$ and $m+n\geq 2k\geq 6$, we have
    \begin{equation}
    \label{Deltaineq}
        \Delta(F_k(K_{m,n}))\leq mn.
    \end{equation}
\end{lemma}
\begin{proof}
    Consider the expression in (\ref{degchi}) as a polynomial in $\ell$; denote this polynomial by $p(\ell)$. Observe that $p(\ell)$ is quadratic with a positive leading coefficient. Hence, for $0\leq \ell\leq \min(m,k)$, the maximum of $p(\ell)$ occurs either at $\ell=0$ or $\ell=\min(m,k)$. It is now routine to verify that $p(\ell)\leq mn$ in either case. The result follows.
\end{proof}

\begin{lemma}
    For positive integers $m,n,k$ such that $2\leq m\leq n$ and $m+n\geq 2k\geq 6$, we have
    \begin{equation}
    \label{gammalower}
        \frac{1}{1+mn} \binom{m+n}{k} \leq \gamma(F_k(K_{m,n})).
    \end{equation}
\end{lemma}

\begin{proof}
    Combine (\ref{deltag}), (\ref{vertcalX}) to obtain
    \begin{equation}
    \label{Deltakineq}
        \frac{1}{1+\Delta(F_k(K_{m,n}))}\binom{m+n}{k}\leq \gamma(F_k(K_{m,n})).
    \end{equation}
    Combine (\ref{Deltaineq}), (\ref{Deltakineq}) to obtain the result.
\end{proof}

Our next goal is to prove the right inequality of (\ref{intro2}).

\begin{lemma}
    For positive integers $m,n,k$ such that $2\leq m\leq n$ and $m+n\geq 2k\geq 6$ we have
    \begin{equation}
    \label{alphaupper}
        \alpha (F_k(K_{m,n}))\leq\frac{n}{m+n-k+1}\binom{m+n}{k}.
    \end{equation}
\end{lemma}

\begin{proof}
    Recall that $\vert X\vert=m+n$. Combine (\ref{alphaG}), (\ref{alphaK}) and simplify to obtain the result.
\end{proof}

\begin{corollary}
    For positive integers $m,n,k$ such that $2\leq m\leq n$ and $m+n\geq 2k\geq 6$ we have
    \begin{equation}
    \label{gammaupper}
        \gamma(F_k(K_{m,n}))\leq\frac{n}{m+n-k+1}\binom{m+n}{k}.
    \end{equation}
\end{corollary}

\begin{proof}
    Combine (\ref{alphaupper}) and Lemma \ref{gammaalpha}.
\end{proof}

We now prove Theorem \ref{main2}.
\begin{proof}[Proof of Theorem \ref{main2}]
    Combine (\ref{gammalower}) and (\ref{gammaupper}) to obtain the result.
\end{proof}

\section{Appendix}
In this section, we illustrate the equitable partition described in Remark \ref{remark}.

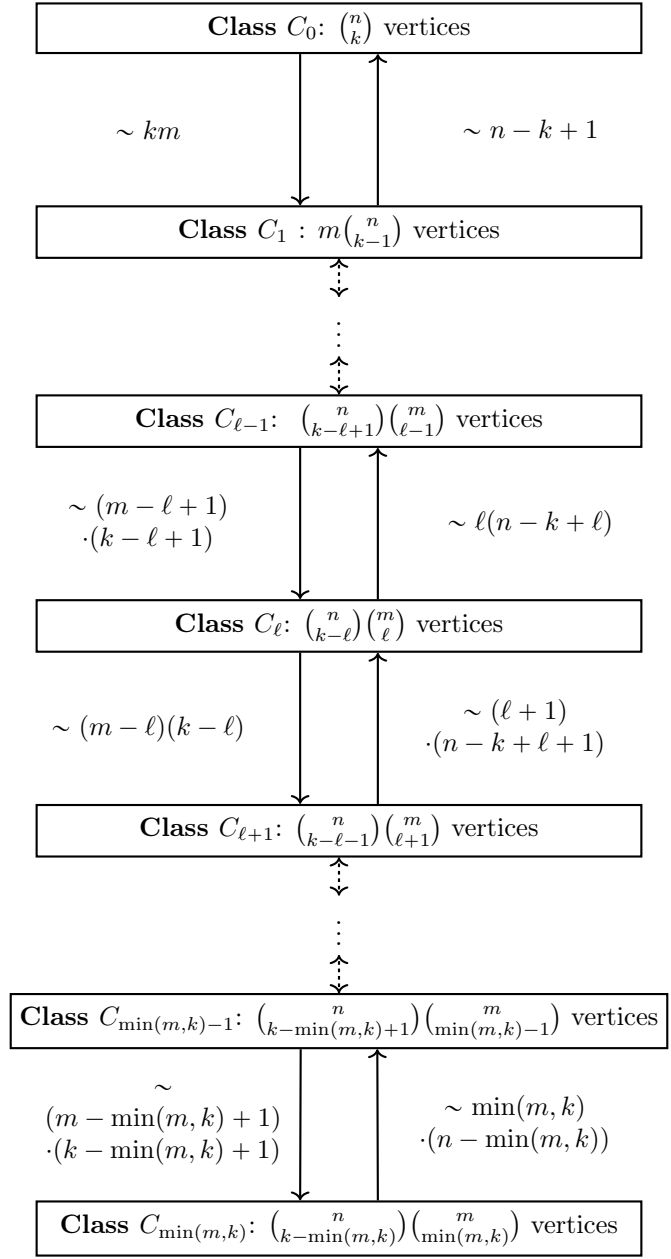
\begin{figure}[h!]
    \centering
\begin{tikzpicture}[node distance= 20mm, minimum width=8cm, thick, main/.style = {draw, rectangle}]

\node[main, align=center] (1) {\textbf{Class $C_0$}:  $\binom{n}{k}$ vertices};

\node[main, align=center] (2) [below= of 1] {\textbf{Class $C_1$ }: $m\binom{n}{k-1}$ vertices};

\draw[->] ([xshift=35mm] 1.south west) to node[left, xshift=20mm] {$\sim km$} ([xshift=35mm] 2.north west);

\draw[->] ([xshift=-35mm] 2.north east) to node[right, xshift=-20mm] {$\sim n-k+1$}([xshift=-35mm] 1.south east);

\node[main, draw=none] (3) [below =5mm of 2] {$\vdots$};

\draw[->, dotted]  (2) to (3);
\draw[->, dotted]  (3) to (2);

\node[main, align=center] (4) [below =5mm of 3] {\textbf{Class $C_{\ell-1}$}: \ $\binom{n}{k-\ell+1}\binom{m}{\ell - 1}$ vertices};

\draw[->, dotted]  (3) to (4);
\draw[->, dotted]  (4) to (3);

\node[main, align=center] (5) [below= of 4] {\textbf{Class $C_\ell$}: $\binom{n}{k-\ell}\binom{m}{\ell}$ vertices};

\draw[->, align=center] ([xshift=35mm] 4.south west) to node[left, xshift=20mm] {$\sim (m-\ell+1)$ \\ $\cdot(k-\ell+1)$} ([xshift=35mm] 5.north west);

\draw[->, align=center] ([xshift=-35mm] 5.north east) to node[right, xshift=-20mm] {$\sim \ell(n-k+\ell)$} ([xshift=-35mm] 4.south east);

\node[main, align=center] (6) [below= of 5] {\textbf{Class $C_{\ell+1}$}: $\binom{n}{k-\ell-1}\binom{m}{\ell+1}$ vertices};

\node[main, draw=none] (7) [below =5mm of 6] {$\vdots$};

\draw[->, dotted]  (6) to (7);
\draw[->, dotted]  (7) to (6);

\draw[->, align=center] ([xshift=35mm] 5.south west) to node[left, xshift=20mm] {$\sim (m-\ell)(k-\ell)$} ([xshift=35mm] 6.north west);

\draw[->, align=center] ([xshift=-35mm] 6.north east) to node[right, xshift=-22mm] {$\sim (\ell+1)$\\$\cdot (n-k +\ell+1)$} ([xshift=-35mm] 5.south east);

\node[main, align=center] (8) [below=5mm of 7] {\textbf{Class $C_{\min (m,k)-1}$}: $\binom{n}{k-\min(m,k)+1}\binom{m}{\min(m,k)-1}$ vertices};

\draw[->, dotted]  (7) to (8);
\draw[->, dotted]  (8) to (7);

\node[main, align=center] (9) [below = of 8] {\textbf{Class $C_{\min(m,k)}$}: $\binom{n}{k-\min(m,k)}\binom{m}{\min(m,k)}$ vertices};

\draw[->, align=center] ([xshift=38.5mm] 8.south west) to node[left, xshift=22mm] {$\sim $ \\ $(m-\min(m,k)+1)$ \\$\cdot(k-\min(m,k)+1)$} ([xshift=35mm] 9.north west);

\draw[->, align=center] ([xshift=-35mm] 9.north east) to node[right, xshift=-22mm] {$\sim \min(m,k)$\\$\cdot(n-\min(m,k))$} ([xshift=-38.5mm] 8.south east);
\end{tikzpicture}
\caption{For $0<\ell\leq \min(m,k)$, the arrow from $\mathcal{C}_{\ell}$ to $\mathcal{C}_{\ell-1}$ (resp. $\mathcal{C}_{\ell-1}$ to $\mathcal{C}_{\ell}$) and its corresponding number represent the number of vertices in $\mathcal{C}_{\ell-1}$ that are adjacent to each vertex in $\mathcal{C}_{\ell}$ (resp. number of vertices in $\mathcal{C}_{\ell}$ that are adjacent to each vertex in $\mathcal{C}_{\ell-1}$).}
    \label{class}
\end{figure}








\section*{Acknowledgements}
This project was a part of the SMALL 2026 REU program, which was funded by the U.S. National Science Foundation (Grant DMS2241623). The authors would also like to thank Williams College for all the support.

\section*{Declarations}
\subsection*{Data Availability Statement}
No datasets were generated or analyzed during the current study.

\begin{multicols}{2}
Nathaniel Chinman\\
Brown University\\
Providence, RI 02912 USA\\
email: nathaniel\_chinman@brown.edu
\columnbreak 

Antonio Cortijo\\
University of California, Berkeley\\
2500 Bancroft Way\\
Berkeley, CA 94704 USA\\
email: acortijorodgers@berkeley.edu
\end{multicols}

\begin{multicols}{2}
Beatrice Pedroni\\
Williams College \\
880 Main St \\
Williamstown, MA 01267-2680 USA \\
email: bp11@williams.edu
\columnbreak 

Ian Seong\\
Department of Mathematics\\
Williams College \\
18 Hoxsey St \\
Williamstown, MA 01267-2680 USA \\
email: is11@williams.edu
\end{multicols}

Lucy Stenger\\
Williams College \\
880 Main St \\
Williamstown, MA 01267-2680 USA \\
email: lhs3@williams.edu
\end{document}